\documentclass[a4paper,reqno, 10pt]{amsart}
\usepackage{amsmath,amssymb,amsfonts,amsthm,mathrsfs}
\usepackage{tikz}
\usepackage{mathtools}
\usepackage{verbatim,wasysym,cite}
\usepackage{physics}
\usepackage[latin1]{inputenc}
 \mathtoolsset{showonlyrefs=true}
\usepackage{microtype}
\usepackage{color,enumitem,graphicx}
\usepackage[colorlinks=true,urlcolor=blue, citecolor=red,linkcolor=blue,
linktocpage,pdfpagelabels, bookmarksnumbered, bookmarksopen]{hyperref}
\usepackage[english]{babel}
\renewcommand{\epsilon}{{\varepsilon}}
\usepackage{enumitem}
\numberwithin{equation}{section}
\newtheorem{theorem}{Theorem}[section]
\newtheorem{lemma}[theorem]{Lemma}
\newtheorem{remark}[theorem]{Remark}
\newtheorem{Def}[theorem]{Definition}
\newtheorem{proposition}[theorem]{Proposition}
\newtheorem{corollary}[theorem]{Corollary}

\newcommand{\C}{\mathbb C}

\newcommand{\R}{\mathbb R}

\def\({\left(}
\def\){\right)}
\def\<{\left\langle}
\def\>{\right\rangle}

\usepackage{geometry}
\newtheorem{assumption}{Assumption}

\begin{document}
	\title{A Paley-Wiener type uniqueness result for the electromagnetic Schr\"odinger equation}
\author[Y. Song]{Yilin Song}
 \address{ Yilin Song
\newline Institute of Applied Physics and Computational Mathematics,
 		Beijing, 100088, China.}
\email{songyilin21@gscaep.ac.cn}
\author[Y. Wang]{Ying Wang}
\address{Ying Wang\newline
  BCAM - Basque Center for Applied Mathematics, 48009, Bilbao, Spain}
\email{ywang@bcamath.org}
 \author[J. Zheng]{Jiqiang Zheng}
 \address{Jiqiang Zheng
 		\newline \indent Institute of Applied Physics and Computational Mathematics,
 		Beijing, 100088, China.
 		\newline\indent
 		National Key Laboratory of Computational Physics, Beijing 100088, China}
 \email{zhengjiqiang@gmail.com}	
 \author[R Zhou]{Ruihan Zhou}
 \address{ Ruihan Zhou
\newline Institute of Applied Physics and Computational Mathematics,
 		Beijing, 100088, China.}
\email{19210180085@fudan.edu.cn}
		\maketitle
		
	\noindent \textbf{Abstract:} In this paper, we establish a Paley-Wiener type uncertainty principle for Schr\"odinger equations with bounded  electric and magnetic  potentials,
	\begin{align*}
		i\partial_tu+\Delta_Au+V(t,x)u=0,\,\,u(0,x)=u_0(x),
	\end{align*}
	where $\Delta_A=(\nabla-iA)^2$ denotes the magnetic Schr\"odinger operator.
	Specifically, under suitable assumptions on $A$ and $V$, we show that if a solution 
	$u$
	exhibits   linear exponential decay and support property in one spatial direction at times $t=0$ and $t=1$ respectively, then 
	$u$ must vanish identically.  This result extends the theorem of Kenig-Ponce-Vega [Ann. Sci. \'Ec. Norm. Sup\'er. (4) 47 (2014), 539-557] to the case $A\neq0$. 
   We overcome the difficulty brought by the magnetic potential which breaks the translation invariance in the leading term of
Hamiltonian $H=\Delta_A+V$. As a direct consequence, we also obtain a uniqueness result for a class of semi-linear Schr\"odinger equation with electromagnetic potentials.

	\begin{center}
		\begin{minipage}{145mm}
			{ \small {{\bf Key Words:}  Paley-Wiener uncertainty principle; the electromagnetic Schr\"odinger equation; Carleman estimate; Uniqueness.}
				{}
			}\\
			{ \small {\bf AMS Classification:}
				{35Q55, 35Q41.}
			}
		\end{minipage}
	\end{center}

\section{Introduction}
In this article, we study the uniqueness property of the Cauchy problem
\begin{align}\label{eq1}
	\begin{cases}
		i\partial_tu+\Delta_{A}u+V(t,x)u=0,&(t,x)\in\R\times\R^d,\\
		u(x,0)=u_0(x),
	\end{cases}
\end{align}
where $u:I\times\R^d\to\mathbb C$ is the unknown function and $\Delta_A$ is the  Schr\"odinger operator  with magnetic potentials  $A=A(x):\R^{d}\to\R^d, \,V=V(t,x):\R^{d+1}\to\mathbb C$. More precisely, $\Delta_{A}$ is given by
\begin{align*}
	\Delta_A=(\nabla-iA)^2=\Delta-iA\cdot \nabla-i\nabla\cdot A-|A|^2.
\end{align*}The uniqueness problem  is tightly   connected with the uncertainty principle, which asserts that a nontrivial function cannot be simultaneously localized in both physical and frequency spaces. 
In 1933, Hardy \cite{Hardy} first proved the following fundamental uncertainty principle.
\begin{theorem}[Hardy's uncertainty principle]If $f(x)=\mathcal{O}(e^{-\alpha|x|^2})$ and its Fourier transform has Gaussian decay $\widehat{f}(\xi)=\mathcal{O}(e^{-\beta|\xi|^2})$ with $\alpha\beta>\frac{1}{16}$, then $f=0$.
	
\end{theorem}
As a toy model, the free Schr\"odinger equation can be represented as
\begin{align*}
	u(t,x):= e^{it\Delta }f=\frac{1}{(2\pi t)^\frac d2}\int_{\R^d}e^{i\frac{|x-y|^2}{4t}}f(y)\,\dd y=(2\pi t)^{-\frac d2}e^{i\frac{|x|^2}{4t}}\mathcal{F}\Big(e^{i\frac{|\cdot|^2}{4t}}f\Big)\Big(\frac{x}{2t}\Big).
\end{align*}
Then, by applying the Hardy theorem, we can show that if $u(x,0)=\mathcal O(e^{-\alpha|x|^2}) $ and $u(x,T)=\mathcal O(e^{-\beta|x|^2})$ with  $\alpha\beta>\frac{1}{16T}$, then $u\equiv0$.

For dispersive equation, Zhang \cite{Zhang} established  the unique continuation for 1-D cubic NLS $\partial_{t}u=i(\partial_{x}^2u\pm|u|^{2}u)$  by means of  the inverse scattering method. He showed that  if $u(x,t)=0$ for $(x,t)\in(-\infty,a]\times\{0,1\}$, then $u\equiv0$. Subsequently, Bourgain \cite{Bourgain-IMRN}proved that if a solution to the semilinear Schr\"odinger equation $\partial_{t}u=i(\Delta u\pm|u|^{p-1}u)$ has compact support in a  nontrivial time interval, then $u\equiv0$.   Later, in a series of works by Escauriaza-Kenig-Ponce-Vega \cite{EKPV-CPDE,EKPV-2008MRL,EKPV-JEMS,EKPV-Duke}, they finally proved the following uniqueness property for linear Schr\"odinger equation with bounded electric potential
\begin{theorem}[Unique continuation property,\cite{EKPV-CPDE,EKPV-2008MRL,EKPV-JEMS,EKPV-Duke}]\label{Th-EKPV}
	Let $d\geq1$ and $T>0$ bounded. Assume that $u\in L^\infty([0,T],L^2(\R^d))\cap L^2([0,T],H^1(\R^d))$ is the solution to 
	\begin{align}\label{eq:NLS}
		i\partial_tu+\Delta u-V(t,x)u=0.
	\end{align}
	Under the certain boundedness assumption on $V$ and suppose that 
	$$
	\big\|e^{\alpha|x|^2}u(x,0)\big\|_{L^2(\R^d)}+\big\|e^{\beta|x|^2}u(x,T)\big\|_{L^2(\R^d)}<\infty,
	$$
	then we have $u\equiv0$ if $\alpha\beta>\frac{1}{16T}$.
\end{theorem}For the Schr\"odinger equation with magnetic potential,  Fanelli-Vega \cite{Fanelli-Vega} established the magnetic virial identities, which were later used in Barc\'elo-Fanelli-Guiterrez-Ruiz-Vilela \cite{BFGRV-JFA} and Cassano-Fanelli \cite{Cassano-Fanelli} to prove a unique continuation result under $\alpha\beta>\frac{1}{16T}$, together with certain transversal conditions on magnetic field in dimensions $d\geq3$.

However, there exist many different uncertainty principles beyond the classical one proved by Hardy. Escauriaza-Kenig-Ponce-Vega \cite{EKPV-JLMS} prove the Morgan type uncertainty principle for Schr\"odinger equation \eqref{eq:NLS}. Moreover, they showed that if solution admits the exponential decay
\[
\int_{\R^d}|u(x,0)|^2e^{\alpha^p|x|^p/p}\,\dd x+\int_{\R^d}|u(x,1)|^2e^{\beta^q|x|^q/q}\,\dd x<\infty
\]
with $p\in(1,2]$ and $\frac{1}{p}+\frac{1}{q}=1$, then
$u\equiv0$ under the condition $\alpha\beta>N_p$ for some constant $N_p>0$. This result was recently extended to the  Schr\"odinger equation with magnetic potentials by Huang-Wang\cite{Huang-Wang2025}.   
In addition, Kenig-Ponce-Vega \cite{KPV-14}  established the following Paley-Wiener type uncertainty principle, which can be also interpreted as the limiting case ($p=1$) of the Morgan type uncertainty principle.

\begin{theorem}[\cite{KPV-14}]\label{thm-KPV}
Suppose that $u\in C([0,1],L^{2}(\mathbb{R}^{d}))$ is a solution of \eqref{eq:NLS}. Assume that
	\begin{align*}
		\sup_{t\in[0,1]}\int_{\mathbb{R}^{d}}|u(x,t)|^{2}\,\dd x\leq& K_{1}<\infty,\\
		\int_{\mathbb{R}^{d}}e^{2a_{1}|x_{1}|}|u(x,0)|^{2}\,\dd x=&K_{2}<\infty, \textit{ for some }\,\,a_{1}>0,\\
		\operatorname{supp}u(\cdot,1)\subset\{x\in\mathbb{R}^{d}:x_{1}\leq& a_{2}\}, \textit{ for some }\,\, a_{2}<\infty.
	\end{align*}
We also assume that $V\in L^{\infty}(\mathbb{R}^{d}\times[0,1])$ and 
$\lim\limits_{\rho\to\infty}\|V\|_{L^{1}([0,1],L^{\infty}(\mathbb{R}^{d}\backslash B_{\rho}))}=0,$
where $B_\rho$ denotes a ball centered at origin with radius $\rho$. Then $u\equiv0$.
\end{theorem}
\subsection{Main results}
The main goal of this paper is to study Paley-Wiener type uncertainty principle for the covariant Schr\"odinger flow. Our main  theorem generalizes Theorem \ref{thm-KPV} to the Cauchy problem \eqref{eq1}, allowing for a nontrivial magnetic potential, i.e.   $A\neq0$. Before stating the main theorem, we first introduce the assumptions on the magnetic potential $A$ and the associated field $B=\nabla A-(\nabla A)^\top$.
\begin{assumption}\label{assumption on A and B}Let $B$ be the associated magnetic field and suppose that $B$ satisfies the following assumptions:
	\begin{enumerate}
		\item Suppose that $A\in C_{b}^{1}(\mathbb{R}^{d})$. Also, we assume that $B$  is independent of first component, i.e., $\textbf{e}_{1}^{\top}B(x)=0$, with 
		$\|x^\top B\|_{L^{\infty}}:=M_{B}<\infty$.
		\item We denote $\Psi(x):=x^\top B(x)\in\mathbb{R}^{d}$, $\Xi(x):=\int_{0}^{1}\Psi(sx)\,\dd s\in\mathbb{R}^{d}$. There exist positive $\varepsilon^\prime_0,\varepsilon^{\prime\prime}_0$ such that
		\begin{equation}
			\big\||\Psi|^2\big\|_{L_{x}^{\infty}}\leq\varepsilon_0^\prime,\,\,\|\partial_{j}\Psi_j\|_{L_{x}^\infty}\leq\varepsilon_0^{\prime\prime},\,\,j=2,\cdots,d.
		\end{equation}
		% \begin{equation}
			%     \big\||A|^{2}\big\|_{L_t^1L_x^\infty}\leq \varepsilon_0^\prime,\,\,\big\|\partial_j A_j\big\|_{L^1_t L^\infty_x}\leq\varepsilon_0^{\prime\prime},\,\,j=2,\cdots,d.
			% \end{equation}
	\end{enumerate}
	
\end{assumption}
\begin{remark}

Under these assumptions, the operator $\Delta_A$ is self-adjoint on $L^2$. It is worth noting that our assumptions do not require any smallness conditions on the magnetic potential 
$A$ or the electric potential $V$.

\end{remark}
Now, we can state our main result in the following theorem.
\begin{theorem}\label{thm-P}
	Let $d\geq3$ and  $u\in C([0,1],H^{1}(\mathbb{R}^{d}))$ be a solution to \eqref{eq1}.
	Assume   that  the magnetic potential $A$ and magnetic field $B$ satisfy Assumption  \ref{assumption on A and B}.  Suppose further that
	\begin{align}\label{time interval bound}
		\sup_{t\in[0,1]}\int_{\mathbb{R}^{d}}|u(x,t)|^{2}\,\dd x\leq& K_{1}<\infty,\\\label{one directional exponential decay}
		\int_{\mathbb{R}^{d}}e^{2a_{1}|x_{1}|}|u(x,0)|^{2}\,\dd x=&K_{2}<\infty,\textit{ for some }\,\,a_{1}>0,\\\label{one directional support}
		\operatorname{supp}u(\cdot,1)\subset\{x\in\mathbb{R}^{d}:x_{1}\leq& a_{2}\}, \textit{ for some }\,\, a_{2}<\infty.
	\end{align}
	Moreover, assume that the electric potential satisfies
	$\|V\|_{L^{\infty}(\mathbb{R}^{d}\times[0,1])}=M_{0}<\infty$
	and 
	\begin{align}\lim\limits_{\rho\to\infty}\|V\|_{L^{1}([0,1],L^{\infty}(\mathbb{R}^{d}\backslash B_{\rho}))}=0.\label{bound-V}
	\end{align}
	Then, $u\equiv0$.
\end{theorem}

As a consequence, we obtain uniqueness for the nonlinear electromagnetic Schr\"odinger equation 
\begin{align}\label{semi}
	i\partial_tu+\Delta_Au+Vu=F(u,\bar{u}),
\end{align}
where $F:\mathbb C^2\to\mathbb C$ and $F\in C^k$ with $k>\frac d2$.  We further assume that $F(0)=\partial_uF(0)=\partial_{\bar{u}}F(0)=0$ and \begin{equation}
	|\nabla F(u,\bar{u})|\leq c(|u|^{p_1-1}+|u|^{p_2-1}),\,\, p_1,p_2>1.
\end{equation}
\begin{theorem}\label{thm-non}
	Let $d\geq3$ and $u\in C([0,1],H^k(\R^d))$ be a solution to \eqref{semi} with $k>\frac{d}{2}$. Suppose that $A$ satisfies Assumption \ref{assumption on A and B} and $V$ satisfies  \eqref{bound-V}. Assume that 
\begin{align}\label{time interval bound-non}
		\sup_{t\in[0,1]}\int_{\mathbb{R}^{d}}|u_1(x,t)-u_2(t,x)|^{2}\,\dd x\leq& \widetilde K_{1}<\infty,\\\label{one directional exponential decay-non}
		\int_{\mathbb{R}^{d}}e^{2\widetilde a_{1}|x_{1}|}|u_1(x,1)-u_2(x,1)|^{2}\,\dd x=&\widetilde K_{2}<\infty,\textit{ for some }\,\,\widetilde a_{1}>0,\\\label{one directional support-non}
		\operatorname{supp}\big(u_1(\cdot,0)-u_2(x,0)\big)\subset\{x\in\mathbb{R}^{d}:x_{1}\leq& \widetilde a_{2}\}, \textit{ for some }\,\, \widetilde a_{2}<\infty.
	\end{align} Then we have $u_1\equiv u_2$.
\end{theorem}
\textbf{Notations and structure of the paper.}\,\,
We express $ X \lesssim Y $ or $ Y \gtrsim X $ to denote that $ X \leq CY $ for some absolute constant $ C > 0 $.  We employ $ O(Y) $ to represent any quantity $ X $ such that $ |X| \lesssim Y $. The notation $ X \sim Y $ implies that $ X \lesssim Y \lesssim X $. The term $ o(1) $ is used to describe a quantity that converges to zero. We also denote  $\langle x\rangle=\sqrt{1+|x|^2}$.

The paper is organized as follows. In Section 2, we present some notations associated to the magnetic potential including gauge transformation and Appell transformation. In Section 3, we prove the Carleman estimate for function $h\in C_0^\infty(\R^{d+1})$ and $h\in L_t^2([0,1],H^1(\R^d))$ separately. Finally, in Section 4, we apply these Carleman estimates to prove the uniqueness result stated in Theorem \ref{thm-P} and \ref{thm-non}.

\section{Preliminaries}
In this section, we collect some basic notations  and introduce the gauge and Appell transformation. 

Let $A:=(A_1(x),\cdots, A_d(x)):\R^{d}\to\R^d$ be the magnetic potential, we define the covariant derivative
\begin{align*}
\nabla_A=\nabla-iA,\,\,D_j=\partial_j-iA_j.
\end{align*}The associated magnetic field is given by
\begin{align}\label{def-B}
B(x):=\nabla A(x)-(\nabla A)^{\top}(x):=\big[\partial_jA_k-\partial_kA_j\big]_{j,k=1}^d.
\end{align}
%It is clearly that $B$ is also the vector field and can be understood as  the anti-symmetric gradient of $A$. 
We also denote the vector field $\Psi(x)=x^{\top}B$ by
\begin{align}\label{def-psi}
\Psi_{k}(x):=\sum_{j=1}^{d}x_jB_{jk}(x),\,\,k=1,\cdots,d.
\end{align}

Next, we introduce the notation of gauge invariant. Let $u$ be the solution to the following covariant flow
\begin{align}
\partial_tu=i(\Delta_A u+V(t,x)u+F(t,x)).\label{2.1}
\end{align}
For the gauge transform $\tilde{A}=A+\nabla \varphi$ with $\varphi:\R^d\to\R$, if  $v=e^{i\varphi}u$ solves
\begin{align*}
\partial_tv=i(\Delta_{\tilde{A}}v+V(t,x)v+e^{i\varphi}F(t,x)),
\end{align*}
we call \eqref{2.1} is invariant under the gauge transform.
Moreover, it holds $\Delta_A(u)=e^{-i\varphi}\Delta_{\tilde A}(e^{i\varphi}u)$. For convenience, we denote $\nabla_{\tilde A}=\nabla-i\tilde{A}$ and its component can be written as $\tilde{D}_j=\partial_j-i\tilde{A}_j$ for $j=1,\cdots,d$.

\begin{Def}
The covariant operator $\nabla-iA$ is called in the Cronstr\"om gauge if $x\cdot A(x)=0$ holds for any $x\in\R^d$.
\end{Def}
Next, we provide a lemma which connects the Cronstr\"om gauge  with the potential term. The detailed proof can be found in \cite{BFGRV-JFA,Iwatsuka}.
\begin{lemma}\label{gauge}
Let $A:\R^d\to\R^d$ be the magnetic potential and  $B\in\mathcal{M}_{d\times d}(\R)$ be the associated magnetic field. %Suppose that the definition of  $B$ and $\Psi$ are the  same as above. 
Assume that the following quantities are bounded for almost every $x\in\R^d$,
\begin{align}\label{bound A}
	\Big|\int_0^1A(sx)\,ds\Big|<\infty,\,\,\Big|\int_0^1\Psi(sx)\,dx\Big|<\infty.
\end{align}
Then for scalar function $\varphi(x)=x\cdot \int_0^1A(sx)\,ds$, we have the following:
\begin{gather*}
	\tilde A(x)=A-\nabla \varphi=-\int_0^1\Psi(sx)\,ds,\\
	x^{\top}\nabla A(x)=-\Psi(x)+\int_0^1\Psi(sx)\,ds.
\end{gather*}
\end{lemma}
\begin{remark}
    Under the assumption $A\in C_{b}^{1}(\mathbb{R}^d)$, the condition \eqref{bound A} holds automatically.
\end{remark}
\begin{remark}
If the magnetic field $B$ satisfies $\mathbf{e}_1\cdot B=0$, then we have $e_1\cdot \tilde A=0$ with $\tilde A$ given  as in  Lemma \ref{gauge}.
\end{remark}

\begin{corollary}
Under the same assumption as Lemma \ref{gauge}, we have the following transversal conditions,
$
x\cdot\tilde{A}(x)=0,\,\, x\cdot (x^{\top}\nabla \tilde A(x))=0.
$
\end{corollary}
We note that the gauge transform preserves the self-adjointness of $-\Delta_A-V_1$.
\begin{lemma}[Self-adjointness of gauge operator,\cite{BFGRV-JFA}]
Let $A=A(x)=(A_1(x),\dots,A_d(x)):\R^d\to\R^d$, $V_1:\R^d\to\R$ and denote by $ B=\nabla A-(\nabla A)^{\top}$.  For $d\geq2$, assume that  $
\int_0^1 A(sx)\,ds\in\R^d
$
is finite for almost every $x\in\R^d$. Furthermore, we assume that
$V_1\in L^\infty$  and $	\Psi\in L^\infty$. Then for
quadratic form
\begin{align*}
	\widetilde q(\varphi,\psi)
	:=
	\int\nabla_{\widetilde A}\varphi\cdot\overline{\nabla_{\widetilde A}\psi}\,dx
	+\int V_1\varphi\overline\psi\,dx,
\end{align*}
$\widetilde q$ is the form associated to a unique self-adjoint operator
$H_{\widetilde A}=-\Delta_{\widetilde A}-V_1(x)$, with form domain $H^1(\R^d)$.
\end{lemma}
Finally, we recall the Appell transform adapted to the covariant Schr\"odinger flow.
\begin{lemma}[\cite{BFGRV-JFA}]\label{Appell}
Let $A(t,x)=(A_1(t,x),\cdots,A_d(t,x)):\R^{d+1}\to\R^d$, $V=V(t,x):\R^{d+1}\to\C$ and $F=F(t,x):\R^{d+1}\to\C$. Suppose that $u(t,x)$ is a solution to 
\begin{align*}
	\partial_tu=i(\Delta_A u+V(t,x)u+F(t,x)).
\end{align*}
Under the following transformation
\begin{align*}
	\tilde{u}(t,x)=\Big(\frac{\sqrt{\alpha\beta}}{\alpha(1-t)+\beta t}\Big)^\frac{d}{2}u\Big(\frac{\sqrt{\alpha\beta}x}{\alpha(1-t)+\beta t},\frac{\beta t}{\alpha(1-t)+\beta t}\Big)e^{\frac{(\alpha-\beta)|x|^2}{4i(\alpha(1-t)+\beta t)}},\,\,\alpha,\beta>0,
\end{align*}
then  $\tilde{u}(t,x)$ solves the modified covariant Schr\"odinger equation
\begin{align*}
	\partial_t\tilde{u}=i\Big(\Delta_{\tilde{A}}\tilde{u}+\frac{(\alpha-\beta)\tilde{A}\cdot x}{\alpha(1-t)+\beta t}\tilde{u}+\tilde{V}(t,x)\tilde{u}+\tilde{F}(t,x)\Big),
\end{align*}
where 
\begin{align*}
	\tilde{A}(t,x)&=\frac{\sqrt{\alpha\beta}}{\alpha(1-t)+\beta t}A\Big(\frac{\sqrt{\alpha\beta}x}{\alpha(1-t)+\beta t},\frac{\beta t}{\alpha(1-t)+\beta t}\Big)\\
	\tilde{V}(t,x)&=\frac{\alpha\beta}{(\alpha(1-t)+\beta t)^2}V\Big(\frac{\sqrt{\alpha\beta}x}{\alpha(1-t)+\beta t},\frac{\beta t}{\alpha(1-t)+\beta t}\Big)\\
	\tilde{F}(t,x)&=\Big(\frac{\sqrt{\alpha\beta}}{\alpha(1-t)+\beta t}\Big)^\frac{2+d}{2}F\Big(\frac{\sqrt{\alpha\beta}x}{\alpha(1-t)+\beta t},\frac{\beta t}{\alpha(1-t)+\beta t}\Big)e^{\frac{(\alpha-\beta)|x|^2}{4i(\alpha(1-t)+\beta t)}}.
\end{align*}
\end{lemma}

\section{Carleman estimates for covariant Schr\"odinger flow}
In this section, we establish a Carleman estimate by adapting the strategy introduced in \cite{KPV-14}. 
\begin{theorem}\label{classical Carleman}
Suppose that $R$ is a  parameter and  $\varphi(t):[0,1]\to\mathbb{R}$ is a smooth function. Denote $x=(x_{1},x_{2},\cdots,x_{d})$, then there exists a constant $C_{1}=C_{1}(d,\|\varphi\|_{C^{2}_{b}([0,1])})>0$ and for some $x_{0,1}\in\R$ such that 
\begin{equation}\label{carleman estimate pw}
	\frac{\sigma^{\frac{3}{2}}}{R^{2}}\Big\|e^{\sigma|\frac{x_{1}-x_{0,1}}{R}+\varphi(t)|^{2}}h\Big\|_{L^{2}([0,1]\times\mathbb{R}^{d})}\leq C\|e^{\sigma|\frac{x_{1}-x_{0,1}}{R}+\varphi(t)|^{2}}(i\partial_{t}+\Delta_{A})h\|_{L^{2}([0,1]\times\mathbb{R}^{d})}
\end{equation}
where $\Delta_A=(\nabla-iA)^2$ with $A:\R^{d+1}\to\R^d$, $\partial_{t}A_{1}\in L^{\infty}_{t,x}$, $\sigma\geq C_{1}R^{2}$ and $h\in C^{\infty}_{c}(\mathbb{R}^{d
	+1})$ with support on 
\begin{equation}\label{compact-support}
	\Big\{(x,t)=(x_{1},\cdots,x_{d},t)\in\mathbb{R}^{d+1}:\Big|\frac{x_{1}-x_{0,1}}{R}+\varphi(t)\Big|\geq1\Big\}.
\end{equation}
\end{theorem}
%\begin{remark}
%    We remark that the magnetic potential appeared in the following theorem is not exactly the same potential as in \eqref{eq1}. In the proof of Theorem \ref{thm-P}, when we apply this Carleman estimate, the magnetic potential we actually use is obtained after applying a scaling transformation and the Appell transformation to the magnetic potential in \eqref{eq1}.
%\end{remark}
\begin{proof}
Let  $f(x,t)=e^{\psi}h$ with $\psi(x,t)=\sigma|\frac{x_{1}-x_{0,1}}{R}+\varphi(t)|^{2}$, then the Carleman estimate reduces to showing
\begin{equation}
	\frac{\sigma^{\frac{3}{2}}}{R^{2}}\|f\|_{L^{2}([0,1]\times\mathbb{R}^{d})}\leq C\|e^{\psi}(\partial_{t}+i\Delta_{A})(e^{-\psi}f)\|_{L^{2}([0,1]\times\mathbb{R}^{d})}.
\end{equation}
% For equation
%     \begin{equation}
	%         \partial_{t}u=i\Delta_{A}u-iV(x,t)u-iF(x,t),
	%     \end{equation}
%   by change of variables $v=e^{\psi}u$ where $\psi$ is a smooth real-valued function, then $v$ satisfies 
%     \begin{align}
	%         \partial_{t}v=&ie^{\psi}\Delta_{A}(e^{-\psi}v)-iV(x,t)v-ie^{\psi}F(x,t)\\=&  
	%         (\mathcal{S}+\mathcal{A})v-iV(x,t)v-ie^{\psi}F(x,t),
	%     \end{align}
%     where
Moreover, for a general functions $f$, one has
\begin{align}
	e^{\psi}(\partial_{t}+i\Delta_{A})(e^{-\psi}f)=\partial_{t}f-(\mathcal{S}+\mathcal{A})f,
\end{align}
where 
$\mathcal{S}=-i(\Delta_{x}\psi+2\nabla_{x}\psi\cdot\nabla_{A})+\psi_{t}$ and $
\mathcal{A}=i(\Delta_{A}+|\nabla_{x}\psi|^{2}).$
One can verify that
\begin{equation}\begin{aligned}
		&\|e^{\psi}(\partial_{t}+i\Delta_{A})e^{-\psi}f\|^{2}_{L^{2}([0,1]\times\mathbb{R}^{d})}=
		\|\partial_{t}f-(\mathcal{S}+\mathcal{A})f\|_{L^{2}([0,1]\times\mathbb{R}^{d})}^{2}\\=&\|(\partial_{t}-\mathcal{A})f\|_{L^{2}([0,1]\times\mathbb{R}^{d})}^{2}-(\partial_{t}f-\mathcal{A}f,\mathcal{S}f)+\|\mathcal{S}f\|^{2}_{L^{2}([0,1]\times\mathbb{R}^{d})}-(\mathcal{S}f,\partial_{t}f-\mathcal{A}f)\\
		\geq&\int_{\mathbb{R}^{d+1}}\big(\mathcal{S}_{t}f\bar{f}+[\mathcal{S},\mathcal{A}]f\bar{f}\big)\,\dd x\dd t.\label{commutator}
	\end{aligned}
\end{equation}
By direct computation, we have 
\begin{gather}
	\nabla_{x}\psi=2\sigma\Big(\frac{x_{1}-x_{0,1}}{R}+\varphi(t)\Big)\frac{1}{R}\mathbf{e}_{1},\,\,
	(\operatorname{Hess}\psi)_{jk}=\frac{2\sigma}{R^{2}}\delta_{1j}\delta_{1k},\\
	\partial_{t}\nabla_{x}\psi=\frac{2\sigma}{R}\varphi'(t)\textbf{e}_{1},\,\,
	\partial_{t}\psi=2\sigma\Big(\frac{x_{1}-x_{0,1}}{R}+\varphi(t)\Big)\varphi^{\prime}(t),\\
	\partial_{tt}\psi=2\sigma(\varphi^{\prime}(t))^{2}+2\sigma\Big(\frac{x_{1}-x_{0,1}}{R}+\varphi(t)\Big)\varphi^{\prime\prime}(t).
\end{gather}
Hence,
\begin{align}
	\mathcal{S}_{t}=
	%&2\Big(\operatorname{Im}\nabla_{x}\partial_{t}\psi\cdot\nabla_{A}-\nabla_{x}\psi\cdot A_{t}\Big)+\partial_{tt}\psi\\=
	&2\Big(\operatorname{Im}\frac{2\sigma}{R}\varphi'(t)D_{1}-2\sigma\Big(\frac{x_{1}-x_{0,1}}{R}+\varphi(t)\Big)\cdot \partial_{t}A_{1}\Big)+2\sigma(\varphi^{\prime}(t))^{2}+2\sigma\Big(\frac{x_{1}-x_{0,1}}{R}+\varphi(t)\Big)\varphi^{\prime\prime}(t).
\end{align}
Moreover, we have
$$
\int_{0}^{1}\int_{\mathbb{R}^{d}}\mathcal{S}_{t}f\bar{f}\,\dd x\dd t=\frac{4\sigma}{R}\operatorname{Im}\int_{0}^{1}\int_{\mathbb{R}^{d}}\varphi^{\prime}(t)D_{1}f\bar{f}\,\dd x\dd t-4\sigma\int_{0}^{1}\int_{\mathbb{R}^{d}}\Big(\frac{x_{1}-x_{0,1}}{R}+\varphi(t)\Big)\partial_{t}A_{1}|f|^{2}\,\dd x\dd t$$
$$+\int_{0}^{1}\int_{\mathbb{R}^{d}}2\sigma(\varphi^{\prime}(t))^{2}|f|^{2}\,\dd x\dd t+2\sigma\int_{0}^{1}\int_{\mathbb{R}^{d}}\Big(\frac{x_{1}-x_{0,1}}{R}+\varphi(t)\Big)\varphi^{\prime\prime}(t)|f|^{2}\,\dd x\dd t=:\sum_{j=1}^{4}Z_{j},
$$
and
\begin{align*}
& \int_{0}^{1}\int_{\mathbb{R}^{d}}[\mathcal{S},\mathcal{A}]f\bar{f}\,\dd x\dd t\\
=&\frac{8\sigma}{R^{2}}\int_{0}^{1}\int_{\mathbb{R}^{d}}|D_{1}f|^{2}\,\dd x\dd t+\frac{32\sigma^{3}}{R^{4}}\int_{0}^{1}\int_{\mathbb{R}^{d}}|\frac{x_{1}-x_{1,0}}{R}+\varphi(t)|^{2}|f|^{2}\,\dd x\dd t+\frac{4\sigma}{R}\operatorname{Im}\int_{0}^{1}\int_{\mathbb{R}^{d}}\bar{f}\varphi'(t)D_{1}f\,\dd x\dd t\\
&-\frac{8\sigma}{R}\int_{0}^{1}\int_{\mathbb{R}^{d}}\Big(\frac{x_{1}-x_{1,0}}{R}+\varphi(t)\Big)\textbf{e}_{1}B\cdot\overline{\nabla_{A}f}f\,\dd x\dd t=\sum_{j=5}^{8}Z_{j}.
\end{align*}

For $Z_{1}$, using Cauchy-Schwarz's inequality and \eqref{compact-support}, we have
$$
Z_{1}\geq-8\sigma\|\varphi^{\prime}\|^{2}_{L^{\infty}([0,1])}\int_{0}^{1}\int_{\mathbb{R}^{d}}|f|^{2}\,\dd x\dd t-\frac{2\sigma}{R^{2}}\int_{0}^{1}\int_{\mathbb{R}^{d}}|D_{1}f|^{2}\,\dd x\dd t
$$
$$\quad\qquad\quad\quad\geq-8\sigma\|\varphi^{\prime}\|^{2}_{L^{\infty}([0,1])}\int_{0}^{1}\int_{\mathbb{R}^{d}}|\frac{x_{1}-x_{1,0}}{R}
+\varphi|^{2}|f|^{2}\,\dd x\dd t-\frac{2\sigma}{R^{2}}\int_{0}^{1}\int_{\mathbb{R}^{d}}|D_{1}f|^{2}\,\dd x\dd t.
$$
For $Z_{2}$, $Z_{3}, Z_{4}$, the H\"older's inequality implies
$$
Z_{2}+Z_{3}+Z_{4}%\geq&-\sigma\Big(\|\varphi^{\prime}\|_{L^{\infty}([0,1])}+\|\varphi^{\prime\prime}\|_{L^{\infty}([0,1])}+\|\partial_{t}A_{1}\|_{L^{\infty}}\Big)\int_{0}^{1}\int_{\mathbb{R}^{d}}\Big|\frac{x_{1}-x_{0,1}}{R}+\varphi(t)\Big||f|^{2}\,\dd x\dd t\\
\geq-\sigma\Big(\|\varphi^{\prime}\|_{L^{\infty}([0,1])}+\|\varphi^{\prime\prime}\|_{L^{\infty}([0,1])}+\|\partial_{t}A_{1}\|_{L^{\infty}}\Big)\int_{0}^{1}\int_{\mathbb{R}^{d}}\Big|\frac{x_{1}-x_{0,1}}{R}+\varphi(t)\Big|^{2}|f|^{2}\,\dd x\dd t.
$$
For $Z_{7}$, since $\textbf{e}_{1}^{\top}B=0$, one has $Z_7=0$.

For $Z_{8}$, combining \eqref{compact-support} with Cauchy-Schwarz's inequality, we get
$$
Z_{8}%\geq& -8\sigma\|\varphi^{\prime}\|^{2}_{L^\infty([0,1])}\int_{0}^{1}\int_{\mathbb{R}^{d}}|f|^{2}\,\dd x\dd t-\frac{2\sigma}{R^{2}}\int_{0}^{1}\int_{\mathbb{R}^{d}}|D_{1}f|^{2}\,\dd x\dd t\\
\geq-8\sigma\|\varphi^{\prime}\|^{2}_{L^\infty([0,1])}\int_{0}^{1}\int_{\mathbb{R}^{d}}\Big|\frac{x_{1}-x_{0,1}}{R}+\varphi(t)\Big|^{2}|f|^{2}\,\dd x\dd t-\frac{2\sigma}{R^{2}}\int_{0}^{1}\int_{\mathbb{R}^{d}}|D_{1}f|^{2}\,\dd x\dd t.
$$
Combining all the estimates, we get
% \begin{align}
	% 	&\int_{0}^{1}\int_{\mathbb{R}^{d}}\mathcal{S}_{t}f\bar{f}+[\mathcal{S},\mathcal{A}]f\bar{f}\,\dd x\dd t\\
	% 	\geq&\Big(\frac{32\sigma^{3}}{R^{4}}-16\sigma\|\varphi'\|^{2}_{L^{\infty}([0,1])}-\sigma\|\varphi^{\prime\prime}\|_{L^{\infty}([0,1])}-\sigma\|\partial_{t}A_{1}\|_{L^{\infty}}\Big)\int_{0}^{1}\int_{\mathbb{R}^{d}}\Big|\frac{x_{1}-x_{0,1}}{R}+\varphi(t)\Big|^{2}|f|^{2}\,\dd x\dd t\\
	% 	&+\frac{4\sigma}{R^{2}}\int_{0}^{1}\int_{\mathbb{R}^{d}}|D_{1}f|^{2}\,\dd x\dd t.
	% \end{align}
\begin{align*}	
&\int_{0}^{1}\int_{\mathbb{R}^{d}}\mathcal{S}_{t}f\bar{f}+[\mathcal{S},\mathcal{A}]f\bar{f}\,\dd x\dd t\\
\geq&\Big(\frac{32\sigma^{3}}{R^{4}}-16\sigma\|\varphi'\|^{2}_{L^{\infty}([0,1])}-\sigma\|\varphi^{\prime\prime}\|_{L^{\infty}([0,1])}-\sigma\|\partial_{t}A_{1}\|_{L^{\infty}}\Big)\int_{0}^{1}\int_{\mathbb{R}^{d}}\Big|\frac{x_{1}-x_{0,1}}{R}+\varphi(t)\Big|^{2}|f|^{2}\,\dd x\dd t\\
&+\frac{4\sigma}{R^{2}}\int_{0}^{1}\int_{\mathbb{R}^{d}}|D_{1}f|^{2}\,\dd x\dd t.
\end{align*}
To obtain the desired lower bound, taking $\sigma$ such that $\sigma\geq C_{1}R^{2}$, where $C_{1}$ to be determined later, we let
\begin{align}
	\frac{8\sigma^{3}}{R^{4}}\geq 16\sigma\|\varphi^{\prime}\|^{2}_{L^{\infty}([0,1])}+\sigma\|\varphi^{\prime\prime}\|_{L^{\infty}([0,1])}+\sigma\|\partial_{t}A_{1}\|_{L^{\infty}}.
\end{align}
It yields that 
\begin{align}
	\sigma^{2}\geq \frac{R^{4}}{8}(16\|\varphi'\|^{2}_{L^{\infty}}+\|\varphi''\|_{L^{\infty}}+\|\partial_{t}A_{1}\|_{L^{\infty}}).
\end{align}
Hence, we get the explicit dependence of constant $C_1=C_1(d,B,\|\partial_{t}A_{1}\|_{L^\infty},\|\varphi^{\prime}\|_{L^{\infty}},\|\varphi^{\prime\prime}\|_{L^{\infty}})$. Putting these estimates together, we arrive at the final Carleman estimate
\begin{align}
	\|e^{\psi}(\partial_{t}+i\Delta_{A})e^{-\psi}f\|_{L^{2}([0,1]\times\mathbb{R}^{d})}^{2}\geq\frac{8\sigma^{3}}{R^{4}}\|f\|_{L^{2}([0,1]\times\mathbb{R}^{d})}^{2}.
\end{align}
\end{proof}
Consequently, we will show that the Carleman estimate holds for functions with compact supports in $(x_1,t)$ directions.
\begin{corollary}\label{Carleman coro}
Suppose that $h\in L^{2}([0,1],H^1(\mathbb{R}^{d}))$ only has a compact support in variables $(x_1,t)$ and satisfies \eqref{compact-support}
%\begin{align}
%	\operatorname{supp}h\subset\Big\{(x,t)\in\mathbb{R}^{d
	%		+1}:\Big|\frac{x_{1}-x_{0,1}}{R}+\varphi(t)\Big|\geq 1\Big\}
%\end{align}
for some $x_{0,1}\in\R$.
In addition, if $(i\partial_{t}+\Delta_{A})h\in L^{2}(\mathbb{R}^{d+1})$, then the Carleman estimates \eqref{carleman estimate pw} holds for $h$.
\end{corollary}
\begin{proof}
Denote $x=(x_{1},x^{\prime})$. Let cutoff function $\eta_{1}\in C_{0}^{\infty}(\mathbb{R})$, $\eta_{1}\geq0$ with  $\operatorname{supp}\eta_{1}\subset\{|x_{1}|<1\}$. Assume another cutoff function $\eta_{2}\in C_{0}^{\infty}(\mathbb{R}^{d-1})$, $\eta_{2}\geq0$ with  $\operatorname{supp}\eta_{2}\subset\{|x'|<1\}$. Moreover, $\eta_1$ and $\eta_2$ obey 
\begin{equation}
	\int_{\mathbb{R}}\eta_{1}(x_{1})\,\dd x_{1}=1, \,\,\int_{\mathbb{R}^{d-1}}\eta_{2}(x^{\prime})\,\dd x^{\prime}=1.
\end{equation}
For $\delta>0$, we define a dilation
\begin{equation}
	K_{\delta}(x,t)=\frac{1}{\delta^{d+1}}\eta_{1}\Big(\frac{t}{\delta}\Big)\eta_{1}\Big(\frac{x_{1}}{\delta}\Big)\eta_{2}\Big(\frac{x'}{\delta}\Big)
\end{equation}
and the mollifier of $h$ can be defined via
$
h_\delta=  \mathcal K_\delta h:=K_{\delta}* h.$

Let $\theta\in C_{0}^{\infty}(\mathbb{R}^{d-1})
$, $\theta(x^{\prime})=1$ for $|x^{\prime}|\leq 1$ with $\operatorname{supp}\theta\subset\{|x'|<2\}$. For large $\ell$, we set 
$
h_{\delta,\ell}(x,t)=\theta\Big(\frac{x^{\prime}}{\ell}\Big)h_{\delta}(x,t).
$
For small $\delta$, we know that
\begin{equation}
	\operatorname{supp}h_{\delta}\subset\bigg\{(x,t):\Big|\frac{x_{1}-x_{0,1}}{R}+\varphi(t)\Big|^{2}\geq\frac{1}{2}\bigg\}.
\end{equation}
Also, we have that $h_{\delta,\ell}\in C_{0}^{\infty}(\mathbb{R}^{d+1})$.

Now,  applying  Carleman estimate \eqref{carleman estimate pw} to $h_{\delta,\ell}$ yields
$$
\frac{\sigma^{\frac{3}{2}}}{R^{2}}\Big\|e^{\sigma\big|\frac{x_{1}-x_{0,1}}{R}+\varphi(t)\big|^{2}}h_{\delta,\ell}\Big\|_{L^{2}(\mathbb{R}^{d+1})}\leq \Big\|e^{\sigma\big|\frac{x_{1}-x_{0,1}}{R}+\varphi(t)\big|^{2}}(i\partial_{t}+\Delta_{A})h_{\delta,\ell}\Big\|_{L^{2}([0,1]\times\mathbb{R}^{d})}.
$$
The direct computation gives
\begin{align}
	(i\partial_{t}+\Delta_{A})h_{\delta,\ell}=i\theta(\frac{x'}{\ell})\partial_{t}h_{\delta}+\theta(\frac{x'}{\ell})\Delta_{A}h_{\delta}+\frac{1}{\ell^{2}}\Delta\theta(\frac{x'}{\ell})h_{\delta}+\frac{2}{\ell}\nabla_{x}\theta(\frac{x'}{\ell})\cdot\nabla_{A}h_{\delta}.
\end{align}
Let $\ell\to\infty$, the $L^{2}_{t,x}$ norm of $\frac{1}{\ell^{2}}\Delta\theta(\frac{x'}{\ell})h_{\delta}+\frac{2}{\ell}\nabla_{x}\theta(\frac{x'}{\ell})\cdot\nabla_{A}h_{\delta}$ tends to $0$. Thus, for $h_{\delta}$, we have Carleman estimates
\begin{align}
	\frac{\sigma^{\frac{3}{2}}}{R^{2}}\|e^{\sigma|\frac{x_{1}-x_{0,1}}{R}+\varphi(t)|^{2}}h_{\delta}\|^{2}_{L^{2}([0,1]\times\mathbb{R}^{d})}\leq \|e^{\sigma|\frac{x_{1}-x_{0,1}}{R}+\varphi(t)|^{2}}(i\partial_{t}+\Delta_{A})h_{\delta}\|_{L^{2}([0,1]\times\mathbb{R}^{d})}.
\end{align}
Next, we study the limit of $h_\delta$ as $\delta\to0$. Since convolution cannot commute with the magnetic Schr\"odinger operator, we should analyze the commutator $\big\|[\Delta_{A},\mathcal K_{\delta}]h\big\|_{L^{2}}$. Recalling that,  
$$
\Delta_{A}u=\Delta u-2iA\cdot\nabla u-i(\nabla\cdot A)u-|A|^{2}u,
$$
% we can obtain
% $$
% 	[\Delta_{A},\mathcal K_\delta]=[-2iA\cdot\nabla-i(\nabla\cdot A)-|A|^{2},\mathcal K_\delta],
% $$
and using $[\Delta,\mathcal K_\delta]=0$, it remains to  estimate $\big\|[-2iA\cdot\nabla-i\nabla\cdot A-|A|^{2},\mathcal K_\delta]h\big\|_{L^{2}}$. By definition, we have 
$$
\big\|[A\cdot\nabla,\mathcal K_\delta]h\big\|_{L^{2}}
\leq\|(A\cdot\nabla)K_{\delta}*h-A\cdot\nabla h\|_{L^2}+\|A\cdot\nabla h-K_{\delta}*(A\cdot\nabla h)\|_{L^{2}}\to0
$$
as $\delta\to0$.
Let  $W=-i(\nabla\cdot A)-|A|^{2}\in L^\infty$ be a multiplication operator, then we have
$$
\big\|[W,\mathcal K_\delta]h\big\|_{L^{2}}
\leq\|W(K_{\delta}*h)-Wh\|_{L^2}+\|Wh-K_{\delta}*(Wh)\|_{L^{2}}  \to0
$$
as $\delta\to0$.
Thus, according to the fact
\begin{equation}
	(i\partial_{t}+\Delta_{A})h_{\delta}=K_{\delta}*(i\partial_{t}+\Delta_{A})h+[\Delta_{A},\mathcal K_\delta]h.
\end{equation}
Let $\phi=|\frac{x_{1}-x_{0,1}}{R}+\varphi(t)|^{2}$, then we deduce
\begin{align}
	\frac{\sigma^{\frac{3}{2}}}{R^{2}}\|e^{\sigma\phi}h_{\delta}\|^{2}_{L^{2}(\mathbb{R}^{d+1})}\leq \|e^{\sigma\phi}K_{\delta}*\big((i\partial_{t}+\Delta_{A})h\big)\|_{L^{2}(\mathbb{R}^{d+1})}
	+\|e^{\sigma\phi}[\Delta_{A},\mathcal K_\delta]h\|_{L^{2}}.
\end{align}
Since $e^{\sigma\phi}\leq C_{\sigma,R}$ on $\operatorname{supp}h_{\delta}$, the commutator estimates and the dominated convergence theorem imply the desired Carleman estimate for $h$.  
\end{proof}
\section{Proof of Theorem \ref{thm-P}}

Our proof will follow the strategy in \cite{KPV-14}. To start with, we recall a linear exponential decay estimate for covaraint Schr\"odinger flow established in \cite{Huang-Wang2025}.

\begin{lemma}\label{Huang}
Suppose that electric potential $\mathbb{V}:\mathbb{R}^{d}\times[0,1]\to \mathbb{C}$ satisfies
\begin{equation}
	\|\mathbb{V}\|_{L^{1}([0,1],L^{\infty}(\mathbb{R}^{d}))}\leq\varepsilon_{0}
\end{equation}
and magnetic potential $\mathbb{A} :\mathbb{R}^{d}\times[0,1]\to\mathbb{R}^{d}$ with $\textbf{e}_{1}\cdot \mathbb{A}(x,t)=0$ satisfies
\begin{equation}
	\||\mathbb A|^{2}\|_{L^{1}([0,1],L^{\infty}(\mathbb{R}^{d}))}\leq\varepsilon_{0}^{\prime},\,\,\|\partial_{j}\Bbb A_{j}\|_{L^{1}([0,1],L^{\infty}(\mathbb{R}^{d}))}\leq\varepsilon_{0}^{\prime\prime},\,\,j=2,\cdots,d.
\end{equation}
Let $u\in C([0,1],L^{2}(\mathbb{R}^{d}))$ be a solution to 
\begin{equation}
	\begin{cases}
		i\partial_{t}u+\Delta_{\Bbb A}u=\mathbb{V}(x,t)u+\mathbb{F}(x,t),\\ u(x,0)=u_{0},
	\end{cases}
\end{equation}
in which $\mathbb{F}(x,t)\in  L^{1}([0,1],L^{2}(\mathbb{R}^{d}))$. Assume for some $\vec{v}\in\mathbb{R}^{d}$ such that
\begin{equation*}
	u(x,0), u(x,1)\in L^{2}(e^{2\vec{v}\cdot x}\dd x), \,\, \mathbb{F}\in L^{1}([0,1],L^{2}(e^{2\vec{v}\cdot x}\dd x)),
\end{equation*}
then there exists $C>0$ independent of $\vec{v}$ such that 
\begin{equation}
	\sup_{t\in[0,1]}\|e^{\vec{v}\cdot x}u(t,\cdot)\|_{L^{2}(\mathbb{R}^{d})}\leq C\Big(\|e^{\vec{v}\cdot x}u(0,\cdot)\|_{L^{2}(\mathbb{R}^{d})}+\|e^{\vec{v}\cdot x}u(1,\cdot)\|_{L^{2}(\mathbb{R}^{d})}+\int_{0}^{1}\|e^{\vec{v}\cdot x}\mathbb{F}(t,\cdot)\|_{L^{2}(\mathbb{R}^{d})}\,\dd t\Big).
\end{equation}
\end{lemma}
With Carleman estimate in hand, we are now in position to prove Theorem \ref{thm-P}.
\begin{proof}[Proof of Theorem \ref{thm-P}]The proof will be divided into six steps. \\
\textbf{Step 1: Linear exponential decay.} In this step, we establish that
\begin{equation}\label{one direction integral}
	\sup_{t\in[0,1]}\int_{\mathbb{R}^{d}}e^{2a_{1}x_{1}}|u(t,x)|^{2}\,\dd x\leq K_{3}.
\end{equation}
Using assumption \eqref{bound-V}, we choose  $\rho>0$  sufficiently large so that
\begin{equation}
	\|V\chi_{|x|\geq\rho}\|_{L^{1}([0,1];L^{\infty}(\mathbb{R}^{d}))}\leq \varepsilon_{0}.
\end{equation}
From  \eqref{time interval bound}, \eqref{one directional exponential decay} and \eqref{one directional support}, we have
\begin{equation}
	\int_{\mathbb{R}^{d}}e^{2a_{1}x_{1}}|u(x,0)|^{2}\,\dd x\leq K_{2},\mbox{ and }\int_{\mathbb{R}^{d}}e^{2a_{1}x_{1}}|u(x,1)|^{2}\dd x \leq e^{2a_{1}a_{2}}K_{1}.
\end{equation}

Let $\varphi$ be the gauge function defined in Lemma \ref{gauge} and set $w=e^{i\varphi}u$. Then $w$ satisfies 
$
i\partial_tw+\Delta_{\tilde A}w-V(t,x)w=0,
$
where $u$ solves \eqref{eq1} and  $\tilde A=A-\nabla\varphi$. By the transversal condition $\mathbf{e}_1^{\top} B=0$ and Lemma \ref{gauge}, we have $e_1\cdot \tilde A=0$. 
Next, we set  $\mathbb{F}(x,t)=-\chi_{|x|\leq\rho}V(x,t)w(x,t)$ and $\Bbb A=\tilde A$. Since
\begin{equation}
	\int_{0}^{1}\|e^{a_{1}x_{1}}\chi_{|x|\leq\rho}Vw\|_{L^{2}(\mathbb{R}^{d})}\,\dd t\leq e^{a_{1}\rho}M_{0}K_{1},
\end{equation}
then  \eqref{one direction integral} is an application of Lemma \ref{Huang} with $\vec{v}=a_1\mathbf{e}_1$ and the fact $\|e^{a_1x_1}w(t,\cdot)\|_{L_x^2}=\|e^{a_1x_1}u(t,\cdot)\|_{L^2_x}$.

\textbf{Step 2: Reduction to a strip.}
In this step, we can reduce the uniqueness on whole space to a narrow strip in $x_1$ direction. 
Let  $\delta=\frac{\varepsilon_{0}}{M_{0}+1}$ be a parameter with $M_{0}:=\|V\|_{L_{t,x}^\infty}$ and $\varepsilon_{0}:=\|V\chi_{|x|\geq\rho}\|_{L_{t}^1L_x^\infty}$. Note that $\delta<1$ and 
\begin{equation}
	\int_{1-\delta}^{1}\|V(\cdot,t)\|_{L_{x}^{\infty}}\,\dd t\leq\delta\|V\|_{L^{\infty}_{t,x}}\leq M_{0}\frac{\varepsilon_0}{M_{0}+1}\leq \varepsilon_{0}.
\end{equation}
By rescaling,  we denote $v(x,t)=u(\delta^{\frac{1}{2}}x,\delta t+1-\delta).$ 
To preserve the form of  the Schr\"odinger equation under the rescaling transform, we introduce the rescaled potentials,
\begin{equation}
	V_{\delta}(x,t)=\delta V(\delta^{\frac{1}{2}}x,\delta t+1-\delta),
	\,\,
	A_{\delta}(x,t)=\delta^{\frac{1}{2}} A(\delta^{\frac{1}{2}}x,\delta t+1-\delta).
\end{equation}
Then, $v$ satisfies
\begin{equation}
	\label{scale-v}\partial_{t}v=i(\Delta_{A_{\delta}}+V_{\delta})v
\end{equation}
and has the property
\begin{equation}\label{support on v1}
	\operatorname{supp}(v(\cdot,1))\subset\{x_{1}\leq m\delta^{-\frac12}\},
\end{equation}
if $\operatorname{supp}u\subset\{x:x_1\leq m\}$.
In addition, we have 
\begin{equation}
	\label{bound-V-delta}\|V_{\delta}\|_{L^{\infty}}\leq\delta M_{0}\leq\varepsilon_{0},\,\,\int_{0}^{1}\|V_{\delta}\|_{L^{\infty}}\,\dd t\leq\varepsilon_{0},
\end{equation}
and by a change of variables, we obtain
\begin{equation}\label{integral of v}
	\int_{\mathbb{R}^{d}}|v(x,t)|^{2}\,\dd x=\frac{1}{\delta^{\frac{d}{2}}}\int_{\mathbb{R}^{d}}|u(y,\delta t+1-\delta)|^{2}\,\dd y\leq K_1\delta^{-\frac d2}.
\end{equation}

Thus, from \eqref{one direction integral}, we get
$$\int_{\mathbb{R}^{d}}e^{2a_{1}x_{1}\delta^{\frac{1}{2}}}|v(x,0)|^{2}\,\dd x\leq K_3\delta^{-\frac{d}{2}}.
$$
In the sequel, we will show that 
\begin{equation}\label{reduction to thin strip}
	\int_{\mathbb{R}^{d-1}}\int_{\frac{m}{2\delta^{1/2}}\leq x_{1}\leq\frac{m}{\delta^{1/2}}}|v(x,1)|^{2}\,\dd x_{1}\,\dd x^\prime=\int_{\mathbb{R}^{d-1}}\int_{\frac{m}{2}\leq x_{1}\leq m}|u(x,1)|^{2}\,\dd x_{1}\dd x^\prime=0,
\end{equation}
 under the assumption \eqref{one directional support} with $a_2=m$ and \eqref{one directional exponential decay}, \eqref{time interval bound} and \eqref{bound-V}. In other word, we claim that under the condition \eqref{one directional exponential decay} and $\operatorname{supp}u(x,1)\subset \{x\in\R^d:x_1\leq m\}$, it holds $u(x,1)\equiv0$ in $x_1\in[\frac{m}{2},m]$. Denote $f(x_1,x^\prime)=u(x_1,x^\prime,1)$ and $g(x_1)=f(x_1-a_2+m,x^\prime)$. We know that $\operatorname{supp} g(x_1)\subset\{x_1\leq m\}$. As a consequence, we have $g(x_1)\equiv0$ for $x_1\in[\frac{m}{2},m]$, and hence $f(x_1,x^\prime)\equiv0$ for $x_1\in[a_2-\frac{m}2,a_2]$. Then replacing $a_2$ by $a_2-\frac{m}{2}$, we can extend the uniqueness region. Repeating this procedure step by step, we obtain $u(x,1)\equiv0$ in  $x_1\leq a_2$. The uniqueness  in region $\{x_1>a_2\}$ can be guaranteed by assumption \ref{one directional support}.  By using the Duhamel formula, we can show the uniqueness for all $t\in[0,1]$.

In the following steps, we will focus on proving the claim \eqref{reduction to thin strip}.

\textbf{Step 3: Appell transformation.} In this step, we introduce an Appell transformation to adjust the exponential decay rate. From Lemma \ref{Appell},  for arbitrary $\alpha,\beta>0$, let $v$ be a solution to \eqref{scale-v}, then the Appell transformation of $v$ solves
% \begin{align*}
	% 	\tilde{v}(x,t)=\Big(\frac{\sqrt{\alpha\beta}}{\alpha(1-t)+\beta t}\Big)^\frac{d}{2}v\Big(\frac{\sqrt{\alpha\beta}x}{\alpha(1-t)+\beta t},\frac{\beta t}{\alpha(1-t)+\beta t}\Big)e^{\frac{(\alpha-\beta)|x|^2}{4i(\alpha(1-t)+\beta t)}}
	% \end{align*}
% is a solution to 
\begin{align}
	\label{tilde-v}\partial_t\tilde{v}=i\Big(\Delta_{\tilde{A}_{\delta}}\tilde v+\frac{(\alpha-\beta)\tilde{A}_{\delta}\cdot x}{\alpha(1-t)+\beta t}\tilde{v}+\tilde{V_\delta}(t,x)\tilde{v}\Big),
\end{align} where  
$\tilde{A}_{\delta},\tilde{V_\delta}$ are defined in Lemma \ref{Appell} and we denote by $\tilde{D}_{1}$ the first component of $\nabla_{\tilde{A}_{\delta}}$.

% For a fixed $\lambda>0$, we take $\alpha(\lambda,\delta)>0$ and $\beta(\lambda,\delta)>0$. We recall that it has been proven that 
% \begin{align*}
	%     \big\|e^{a_1x_1\delta^\frac12}v(x,0)\big\|_{L^2_x}^2\leq \frac{K_3}{\delta^{d/2}}.
	% \end{align*}
From \eqref{support on v1} and \eqref{integral of v}, we obtain
\begin{align}\label{bound-v-1-1}
	\big\|e^{\lambda x_1}v(x,1)\big\|_{L_x^2}^2\leq e^{2\lambda m\delta^{-\frac12}}K_1\delta^{-\frac d2}.
\end{align}
For a given $\lambda>0$,
we choose a triple $(\gamma,\alpha,\beta)$ such that
\begin{equation}
	\label{alphabeta2}
	\gamma=(\lambda\delta^{1/2}a_1)^{1/2},\,\,\,\,\beta=\lambda,\,\,\,\alpha=\delta^{1/2}a_1,
\end{equation} which implies  $
\gamma\big(\frac{\alpha}{\beta}\big)^{1/2}=\delta^{1/2}a_1$ and $\gamma\big(\frac{\beta}{\alpha}\big)^{1/2}=\lambda.
$ Furthermore,  we have 
\[
\| e^{\gamma x_1}\widetilde v(x,0)\|_{L^2_x}=\| e^{\gamma(\alpha/\beta)^{1/2} x_1} v(x,0)\|_{L^2_x}=\| e^{a_1x_1\delta^{1/2}}v(\cdot,0)\|_{L^2_x}\leq K_3^\frac{1}{2}\delta^{-\frac d4}.
\]
Similarly, it holds
$
\| e^{\gamma x_1}\widetilde v(x,1)\|_{L^2_x}\leq {e^{\lambda m/\delta^{1/2}} }K_1^{1/2}{\delta^{-\frac d4}}
$

Using the change of variable,
$
\tau=\frac{\beta t}{\alpha(1-t)+\beta t}
$
and using the boundedness of $V_\delta$, i.e. \eqref{bound-V-delta}, one has
\begin{align*}
	\int_0^1\|\,\tilde {V_\delta}(\cdot,t)\|_{L_x^\infty}dt
	&=\int_0^1\|V_{\delta}(\cdot,\tau)\|_{L_x^\infty}\,d \tau \leq \varepsilon_0.
\end{align*}
Combining with the linear exponential decay estimate with $\mathbb{F}=0$, we deduce
\begin{align}
	\sup_{0\leq t \leq 1}\| e^{\gamma x_1} \widetilde v(\cdot,t)\|_{L^2_x} &\leq {C}_1(d) \Big(K_3^\frac12\delta^{-\frac d4}+K_1^\frac12 e^{\lambda m\delta^{-\frac12}}\delta^{-\frac d4}\Big)\\
	&\leq C(\delta,a_1,K_1,K_3)e^{\lambda m\delta^{-\frac12}}.\label{upper bound gamma x1}
\end{align}
By rescaling and \eqref{integral of v},  we have
\begin{equation}
	\label{est-tilde-v}
	\sup_{0\leq t\leq 1}\| \widetilde v(\cdot,t)\|_{L^2}\leq K_1^\frac{1}{2}\delta^{-\frac{d}{4}}.
\end{equation}
Next, we denote a smooth and convex function $\kappa_0(x_1)\geq0$ by
\begin{align}
	\kappa_0(x_1)=\begin{cases}
		0,&x_1\leq0,\\
		x_1-\frac14,&x_1\geq\frac12.
	\end{cases}
\end{align}
Define 
$
\kappa(x_1)=\langle \kappa_0(x_1)\rangle
$,
the direct computation yields
\begin{align*}
	\kappa^\prime(x_1)=\frac{\kappa_0(x_1)\kappa_0^\prime(x_1)}{\langle \kappa_0(x_1)\rangle},\,\,\kappa^{\prime\prime}(x_1)=\frac{\kappa_0^\prime(x_1)^2+\kappa_0(x_1)\kappa_0^{\prime\prime}(x_1)\langle\kappa_0(x_{1})\rangle^2}{\langle \kappa_0(x_1)\rangle^3}.
\end{align*}
In the region $x_1\geq\frac12$, we have the lower bound $
\kappa^{\prime\prime}(x_1)\geq \langle x_1\rangle^{-3}.$

Now, we provide an upper bound of weighted $L^2$ norm of $\tilde v$.
\begin{proposition}
	Let  $\tilde{v}$ be a solution to \eqref{tilde-v}. Then we have
	\begin{align}\label{upper bound for derivative}
		&\sup_{t\in[0,1]}\|e^{\gamma\kappa(x_1)}\tilde{v}\|_{L^{2}}^{2}+4\gamma\int_{0}^{1}\int_{\mathbb{R}^{d-1}}\int_{x_{1}\geq\frac{1}{2}}t(1-t)\frac{1}{\langle x_{1}\rangle^{3}}\,e^{2\gamma\kappa(x_{1})}|\tilde D_{1}\tilde{v}|^{2}\,\dd x_{1}\dd x^\prime\dd t\\
		\leq&c_{\delta,M_{0},a_{1},d}\,\lambda\,e^{2\lambda m\delta^{-\frac12}}.
	\end{align}
\end{proposition} 
\begin{proof}
	Let ${f}(x,t)=e^{\gamma \kappa(x_1)}\tilde{v}(x,t)$. Then $f$ satisfies the equation
	\begin{equation}
		\partial_{t}f=\mathcal{S}f+\mathcal{A}f+ie^{\gamma\kappa(x_1)}F, \,\, \mathbb{R}^{d}\times[0,1],
	\end{equation}
	where
	\begin{equation}
		\mathcal{S}=-i\gamma(\partial_{x_1}^{2}\kappa+2\partial_{x_{1}}\kappa\tilde D_{1}),
		\,
		\mathcal{A}=i(\Delta_{\tilde{A}_{\delta}}+\gamma|\partial_{x_{1}}\kappa|^{2}),
		\,
		F=\tilde{V}_\delta\tilde{v}.
	\end{equation}
	Similar to \eqref{commutator}, we obtain 
	$$
	\partial_{t}^{2}\big(f,f\big)
	\geq2\partial_{t}\operatorname{Re}(\partial_{t}f-\mathcal{A}f-\mathcal{S}f,f)+2(\mathcal{S}_{t}f+[\mathcal{S},\mathcal{A}]f,f)-\|\partial_{t}f-\mathcal{A}f-\mathcal{S}f\|_{L^{2}_{x}}^{2}.
	$$
	% Notice that 
	% \begin{align}
		% 	\int_0^1t(1-t)\frac{d^2}{dt^2}H(t)\,\dd t=H(1)+H(0)-2\int_0^1 H(t)\,\dd t\leq C\sup_{t\in[0,1]}\|f\|_{L_x^2}^2.
		% \end{align}
	Multiplying $t(1-t)$ in both side of  above estimate and integrating in time, we get
	$$
	2\int_{0}^{1}t(1-t)\int_{\mathbb{R}^{d}}(\mathcal{S}_{t}+[\mathcal{S},\mathcal{A}])f\bar{f}\,\dd x\dd t\leq c_{d}\sup_{t\in[0,1]}\|e^{\gamma\kappa(x_1)}\tilde{v}\|_{L^{2}}^{2}+c_{d}\sup_{t\in[0,1]}\|e^{\gamma\kappa(x_1)}F\|_{L^{2}}^{2}.
	$$
	
From \eqref{alphabeta2} and  $\lambda>1$ sufficiently large, we have $\alpha<\beta$, and therefore
$\|\tilde{V}_{\delta}\|\leq\frac{\beta}{\alpha}\|V_{\delta}\|_{L^{\infty}}\leq\frac{\lambda\delta^{\frac{1}{2}}M_{0}}{a_{1}}.$	
On the other hand, a direct computation yields
\begin{align*}
&2 \int_{0}^{1}t(1-t)\int_{\mathbb{R}^{d}}(\mathcal{S}_{t}+[\mathcal{S},\mathcal{A}])f\bar{f}\,\dd x\dd t\\
=&\int_{0}^{1}t(1-t)\int_{\mathbb{R}^{d}}\Big[-4\gamma\operatorname{Im}\partial_{1}\kappa\partial_{t}A_{1}|f|^{2}+8\gamma\partial_{1}^{2}\kappa |\tilde{D}_{1}f|^{2}-2\gamma\partial_{1}^{4}\kappa|f|^{2}+8\gamma^{3}\partial_{1}^{2}\kappa|\partial_{1}\kappa|^{2}|f|^{2}\\
&-8\gamma\operatorname{Im}(\partial_{1}\kappa)f\textbf{e}_{1}^{\top}B\cdot\overline{\nabla_{A}f}\Big]\,\dd x\dd t.\hspace{40ex}
\end{align*}
	Using  $\textbf{e}_{1}^{\top}B=0$, we get
	\begin{align}\label{co1}
		&\int_{0}^{1}t(1-t)\int_{\mathbb{R}^{d}}8\gamma\partial_{1}^{2}\kappa |\tilde D_{1}f|^{2}+8\gamma^{3}\partial_{1}^{2}\kappa|\partial_{1}\kappa|^{2}|f|^{2}\,\dd x\dd t
		\\\leq&\gamma\,c_{A} \sup_{t\in[0,1]}\|f\|_{L^{2}}^{2}+c_{\delta,M_{0},a_1,d}\lambda\sup_{t\in[0,1]}\|f\|^{2}_{L^{2}}+c_{d} \sup_{t\in[0,1]}\|f\|_{L^{2}}^{2}\\
		\leq &c_{\delta,M_{0},a_{1},d }\lambda\,e^{2\lambda m\delta^{-\frac12}}.
	\end{align}
	By direct calculation, we have
	\begin{align}
		\gamma|\tilde D_{1}f|^{2}=e^{2\gamma\kappa(x_{1})}\Big(\gamma|\tilde D_{1}\tilde{v}|^{2}+2\gamma^{2}\tilde D_{1}\tilde{v}\partial_{1}\kappa\tilde{v}+\gamma^{3}\big(\partial_{1}\kappa(x_1)\big)^{2}|\tilde{v}|^{2}\Big).
	\end{align}
	% By Cauchy-Schwarz's inequality, we obtain
	% \begin{equation}
		% 	2\gamma^{2}\tilde D_{1}\tilde{v}\partial_{1}\kappa\tilde{v}\leq\frac{1}{2}\gamma|\tilde D_{1}\tilde{v}|^{2}+2\gamma^{3}(\partial_{1}\kappa)^{2}|\tilde{v}|^{2}.
		% \end{equation}
	For $\lambda$ large enough and $x_{1}\geq\frac{1}{2}$, one gets 
	\begin{align}\label{lower bound}
		&\gamma\int_{0}^{1}\int_{\mathbb{R}^{d-1}}\int_{x_{1}\geq\frac{1}{2}}t(1-t)\frac{1}{\langle x_{1}\rangle^{3}}e^{2\gamma\kappa(x_{1})}|\tilde D_{1}\tilde{v}|^{2}\,\dd x_1\dd x^\prime\dd t\\ \leq& 4\gamma\int_{0}^{1}\int_{\mathbb{R}^{d-1}}\int_{x_{1}\geq\frac{1}{2}}t(1-t)\partial_{1}^{2}\kappa\,e^{2\gamma\kappa(x_{1})}|\tilde D_{1}\tilde{v}|^{2}\,\dd x_1\dd x^\prime\dd t\\\leq& c_{\delta,M_{0},a_{1},d}\lambda\,e^{2\lambda m\delta^{-\frac12}}.
	\end{align}
	Thus, combining \eqref{upper bound gamma x1} and \eqref{lower bound}, we get the desired estimate \eqref{upper bound for derivative}.
	
\end{proof}
\textbf{Step 4: Upper bounds for energy.} In this part, we introduce the following two quantities with $R$ to be chosen later
\begin{align}
	\Psi(R)=&\int_{2\leq x_{1}\leq \frac{R}{2}}\int_{\frac{3}{8}}^{\frac{5}{8}}|\tilde{v}|^{2}\,\dd t\dd x_{1}\dd x^\prime,\\
	\Theta(R)=&\int_{\mathbb{R}^{d-1}}\int_{\frac{1}{2}<x_{1}<R}\int_{\frac{1}{32}}^{\frac{31}{32}}(|\tilde{v}|^{2}+|\tilde D_{1}\tilde{v}|^{2})\,\dd t\dd x_{1}\dd x^\prime.
\end{align}
First, we provide the  bounds of $\Theta(R)$ and $\Psi(R)$ separately.
\begin{lemma}
	Let $c_{\delta,M}>0$ be a fixed constant depending only on $\delta$ and $m$ given above and $R>0$ sufficiently large, then there holds
	\begin{align}\label{theta-R}
		\Theta(R)&\leq c_{\delta,M}(1+R^3)\lambda e^{2\lambda m\delta^{-\frac12}},
		\\ \liminf_{\lambda\to\infty}\Psi(R)&\geq c_{d}\int_{\mathbb{R}^{d-1}}\int_{0<y_{1}<m\delta^{-\frac{1}{2}}}|v(y,1)|^{2}\,\dd y_{1}\dd y^\prime.\label{psi-R}
	\end{align}
\end{lemma}
\begin{proof}
    By the Appell transformation, we have
	\begin{equation}
		\Psi(R)=\int_{\mathbb{R}^{d-1}}\int_{2\leq x_{1}\leq \frac{R}{2}}\int_{\frac{3}{8}}^{\frac{5}{8}}\Big|(\frac{\sqrt{\alpha\beta}}{\alpha(1-t)+\beta t})^{\frac{d}{2}}v(\frac{\sqrt{\alpha\beta}\,x}{\alpha(1-t)+\beta t},\frac{\beta t}{\alpha(1-t)+\beta t})\Big|^{2}\,\dd t\dd x_1\dd x^\prime.
	\end{equation}
	Introduce the change of variables
	$
	\tau(t)=\frac{\beta t}{\alpha(1-t)+\beta t}.
	$
	On the interval $t\in[\frac{3}{8},\frac{5}{8}]$, we have
	$
	\dd t\sim\frac{\beta}{\alpha}\dd \tau
	$
	and
	$\tau\Big(\frac38\Big),\tau\Big(\frac58\Big)\in\Big(\frac12,1\Big).
	%		\tau\Big(\frac{3}{8}\Big)=\frac{3\beta}{5\alpha+3\beta}\in\Big(\frac{1}{2},1\Big),\,\,\mbox{and  }\,
	%		\tau\Big(\frac{5}{8}\Big)=\frac{5\beta}{3\alpha+5\beta}\in\Big(\frac{1}{2},1\Big).
	$
	Thus, for sufficient large $\lambda$,
	\begin{equation}
		\tau\Big(\frac{5}{8}\Big)-\tau\Big(\frac{3}{8}\Big)=\frac{2\alpha\beta}{(5\alpha+3\beta)(3\alpha+5\beta)}\sim\frac{\alpha}{\beta}.
	\end{equation}
  Futhermore, it holds
	$
	\tau\Big(\frac{5}{8}\Big)>\tau\Big(\frac{3}{8}\Big)\nearrow1$ as $\lambda\to\infty$.
	Now, introducing a new change of variable $y=\sqrt\frac{\alpha}{\beta}x$ such that the integral  is contained in  following interval $
	J=\Big[2\sqrt\frac{\alpha}{\beta},\frac{R}{2}\sqrt{\frac{\alpha}{\beta}}\Big].$
	Since $t\in\big[\frac{3}{8},\frac{5}{8}\big]$,	 the determinant of Jacobian has the constant lower bound. 
	One can verify that  
	\begin{align}
		\Psi(R)\geq c_{d}\frac{\beta}{\alpha}\int_{\mathbb{R}^{d-1}}\int_{J}\int_{I(\alpha,\beta)}|v(y,s)|^{2}\,\dd s\dd y_{1}\dd y^\prime
	\end{align}
	with 
	$I(\alpha,\beta)=\big[\tau\big(\tfrac{3}{8}),\tau\big(\tfrac{5}{8})\big]$. Moreover, $|I(\alpha,\beta)|\sim\frac{\alpha}{\beta}$  for $\beta=\lambda\gg 1.
	$
	Let $R$ be a paramter such that 
	$
	R=\frac{2M\lambda^{\frac{1}{2}}m}{(\delta^{\frac{1}{2}}a_{1})^{\frac{1}{2}}c_{d}},
	$
	and  
	$
	\frac{2}{c_{d}}M\geq{\delta^{-\frac{1}{2}}}
	$.
	Since
	$
	2\sqrt{\frac{\alpha}{\beta}}\to0$ as $\lambda\to\infty,
	$	and 
	$
	\frac{R}{2}\sqrt{\frac{\alpha}{\beta}}\geq\frac{m}{\delta^{\frac{1}{2}}},
	$
	we can get \eqref{psi-R}.
    Combining \eqref{est-tilde-v} and \eqref{upper bound for derivative}, we get \eqref{theta-R},
	which completes the proof of this lemma.
\end{proof}
\textbf{Step 5: Lower bound for energy.} For a fixed small $m>0$, we define 
\begin{equation}
	b:=\int_{\mathbb{R}^{d-1}}\int_{\frac{m}{2\delta^{1/2}}<y_{1}<\frac{m}{\delta^{1/2}}}|v(y,1)|^{2}\,\dd y_{1}\dd y^\prime.
\end{equation}
Using \eqref{support on v1} and
\eqref{psi-R}, we have 
\begin{align}
	\int_{\R^{d-1}}  \int_{2\leq x_{1}\leq\frac{R}{2}}\int_{\frac{3}{8}}^{\frac{5}{8}}|\tilde{v}(x,t)|^{2}\,\dd t\dd x_1\dd x^\prime\geq\frac{b}{2}.
\end{align}
In this step, we aim to prove the lower bound of $\Theta(R)$.
\begin{proposition}
	For $\Theta(R)$ as in \eqref{theta-R}, it possesses a lower bound 
	\begin{equation}
		\Theta(R)\geq\frac{b}{8}c^{3}R^{2}e^{-8cR^{2}}.
	\end{equation}
\end{proposition}
\begin{proof}
	Let $x_{0,1}=\frac{R}{2}$ and $\varphi(t):[0,1]\to\mathbb{R}$ be a smooth cutoff function
	\begin{equation}
		\varphi(t)=\begin{cases}
			\frac{3}{2}-\frac{1}{R},\,\,& t\in[\frac{3}{8},\frac{5}{8}],\\ 0, &t\in[0,\frac{1}{4}]\cup[\frac{3}{4},1],
		\end{cases}
	\end{equation}
	with $0\leq\varphi(t)\leq\frac{3}{2}-\frac{1}{R}$ and $\varphi^{\prime}$, $\varphi^{\prime\prime}$ uniformly bounded  for large $R$. Next, we let $\sigma=cR^2$ and $\theta_{R}\in C^{\infty}(\mathbb{R})$ be such that  $0\leq\theta(x_{1})\leq1$ and 
	\begin{equation}
		\theta_{R}(x_{1})=\begin{cases}
			1,&x_{1}\in(1,R-1),\\ 0,&x_{1}\in(-\infty,\frac{1}{2})\cup(R,+\infty).
		\end{cases}
	\end{equation}
	Also, we denote by $\zeta\in C^{\infty}(\mathbb{R})$ with $0\leq\zeta(x_{1})\leq1$ and
	\begin{equation}
		\zeta(x_{1})=\begin{cases}
			0,&x_{1}<1,\\1,&x_{1}>1+\frac{1}{2R}.
		\end{cases}
	\end{equation}
	With these in hand, we define 
	\begin{equation}
		g(x,t)=\theta_{R}(x_{1})\zeta\Big(\frac{x_{1}-\frac{R}{2}}{R}+\varphi(t)\Big)\tilde{v}(x,t).
	\end{equation}
	
	By the construction of cutoff function $\theta_{R}(x_{1})$ and $\varphi(t)$, $g$ is supported in
	\begin{equation}
		\Big\{(x,t):\frac{1}{2}<x_{1}<R,\frac{1}{32}<t<\frac{31}{32},\Big|\frac{x_{1}-\frac{R}{2}}{R}+\varphi(t)\Big|\geq1\Big\}.
	\end{equation}
	Then it follows from Corollary \ref{Carleman coro}, 
	\begin{align}\label{Carleman-new}
		\frac{\sigma^{\frac{3}{2}}}{R^{2}}\Big\|e^{\sigma|\frac{x_{1}-x_{0,1}}{R}+\varphi(t)|^{2}}g\Big\|_{L^{2}([0,1]\times\mathbb{R}^{d})}\leq C\|e^{\sigma|\frac{x_{1}-x_{0,1}}{R}+\varphi(t)|^{2}}(i\partial_{t}+\Delta_{\tilde A_{\delta}})g\|_{L^{2}([0,1]\times\mathbb{R}^{d})}.
	\end{align}
	If $\frac{3}{2}\leq x_{1}<R-1$ and $\frac{3}{8}\leq t\leq\frac{5}{8}$, then $g(x,t)=\tilde{v}(x,t)$, since $\theta_{R}(x_{1})=1$ and 
	$
	\frac{x_{1}-\frac{R}{2}}{R}+\varphi(t)\geq1+\frac{1}{2R}.
	$	
	For $x_{1}>2$, one observes that $\frac{x_{1}}{R}+1-\frac{1}{R}\geq1+\frac{1}{R}$. 
    Consequently, the left-hand side of  \eqref{Carleman-new} admits the lower bound
	\begin{equation}\label{Z0}
		\frac{\sigma^{3}}{R^{4}}e^{2\sigma(1+\frac{1}{R})^{2}}\int_{\mathbb{R}^{d-1}}\int_{2<x_{1}<R-1}\int_{\frac{3}{8}}^{\frac{5}{8}}|\tilde{v}|^{2}\,\dd t\dd x_{1}\dd x^\prime\geq\frac{b}{2}c^{3}R^{2}e^{2\sigma(1+\frac{1}{R})^{2}}
	\end{equation}
	for $R$ large with $\sigma=cR^{2}$ where $c>C_1$. On the other hand, to estimate the upper bound for right hand side of \eqref{Carleman-new}, we compute $i\partial_{t}g+\Delta_{\tilde A_\delta}g$,
	\begin{align}
		(i\partial_{t}+\Delta_{\tilde{A}_{\delta}})g=&\theta_{R}(x_{1})\zeta\Big(\frac{x_{1}-\frac{R}{2}}{R}+\varphi(t)\Big)\tilde{V}_{\delta}\tilde{v}\\&+\zeta\Big(\frac{x_{1}-\frac{R}{2}}{R}+\varphi(t)\Big)\Big(2\theta_{R}^{\prime}(x_{1})\tilde{D}_{1}\tilde{v}+\tilde{v}\theta_{R}^{\prime\prime}(x_{1})\Big)\\&+i\zeta^{\prime}\varphi^{\prime}\tilde{v}+\zeta^{\prime\prime}\frac{1}{R^{2}}\theta_{R}\tilde{v}+\frac{2}{R}\zeta^{\prime}\theta_{R}\tilde{D}_{1}\tilde{v}\\=:&\mathcal{Z}_{1}+\mathcal{Z}_{2}+\mathcal{Z}_{3}.\label{computation of g}
	\end{align}
	As a consequence of \eqref{integral of v} and \eqref{computation of g}, we have $(i\partial_t+\Delta_{\tilde{A}_{\delta}})g\in L^2(\R^{d+1})$. 
	
	Then, we deal with $\mathcal{Z}_{1}=\theta_{R}(x_{1})\zeta\Big(\frac{x_{1}-\frac{R}{2}}{R}+\varphi(t)\Big)\tilde{V}_{\delta}\tilde{v}$.
	Since $\operatorname{supp}\mathcal Z_{1}\subset\{(x,t):\frac{1}{2}<x_{1}<R, \frac{1}{32}<t<\frac{31}{32}\}$, we can get 
	$	\|\tilde{V}_{\delta}\|_{L^{\infty}}\leq\frac{1024\alpha}{\beta}M_{0}
	$ for $\lambda\gg1$.	Thus, 
	for large $R$, the contribution of  $\mathcal{Z}_{1}$ can be absorbed, yielding
	\begin{equation}\label{Zleft}
		\frac{b}{4}c^{3}R^{2}e^{2\sigma(1+\frac{1}{R})^{2}}\leq c\iint_{\mathbb{R}^{d}\times[0,1]}|\mathcal{Z}_{2}|^{2}e^{2\sigma|\frac{x_{1}-\frac{R}{2}}{R}+\varphi(t)|^{2}}\,\dd x\dd t+c\iint_{\mathbb{R}^{d}\times[0,1]}|\mathcal{Z}_{3}|^{2}e^{2\sigma|\frac{x_{1}-\frac{R}{2}}{R}+\varphi(t)|^{2}}\,\dd x\dd t.
	\end{equation}
	
	Next, we estimate $\mathcal{Z}_{2}$. Notice that  $\mathcal{Z}_{2}$ contains the derivative of $\theta_{R}$, and $\operatorname{supp}\nabla\theta_{R}\subset\big\{x_{1}:x_1\in(\frac{1}{2},1)\cup(R-1,R)\big\}$.
	When $x_{1}\in(\frac{1}{2},1)$, the weight function has the upper bound
	$
	\frac{x_{1}-\frac{R}{2}}{R}+\varphi(t)\leq1,
	$
	and we can deduce that 
	$
	\zeta\Big(\frac{x_{1}-\frac{R}{2}}{R}+\varphi(t)\Big)=0$ for $x_{1}\in(\frac{1}{2},1)$.
	When $x_{1}\in(R-1,R)$, we have 
	$
	\frac{x_{1}-\frac{R}{2}}{R}+\varphi(t)%%%\leq1-\frac{1}{2}+\frac{3}{2}-\frac{1}{R}
	\leq2-\frac{1}{R},
	$
	which implies that
	\begin{align}\label{Z2}
		\int_{\mathbb{R}^{d}\times[0,1]}|\mathcal{Z}_{2}|^{2}e^{2\sigma|\frac{x_{1}-\frac{R}{2}}{R}+\varphi(t)|^{2}}\,\dd x\dd t\leq e^{2\sigma(2-\frac{1}{R})^{2}}\Theta(R).
	\end{align}
	
	It remains to estimate the term $\mathcal{Z}_{3}$, which contains the derivative of $\zeta$. By using the support of $\mathcal{Z}_3$, 
	\begin{equation}
		\operatorname{supp}\mathcal{Z}_{3}\subset \Big\{(x,t):\frac{1}{2}<x_{1}<R,\frac{1}{32}<t<\frac{31}{32},1\leq\frac{x_{1}-\frac{R}{2}}{R}+\varphi(t)\leq1+\frac{1}{2R}\Big\},
	\end{equation}
	we have
	\begin{align}\label{Z3}
		&\int_{\mathbb{R}^{d}\times[0,1]}|\mathcal{Z}_{3}|^{2}e^{2\sigma|\frac{x_{1}-\frac{R}{2}}{R}+\varphi(t)|^{2}}\,\dd x\dd t
		% \\\leq& R^{2}e^{2\sigma(1+\frac{1}{2R})^{2}}\int_{\frac{1}{32}}^{\frac{31}{32}}\int_{\mathbb{R}^{d-1}}\int_{\frac{1}{2}<x_{1}<R}|\tilde{v}|^{2}\,\dd x_{1}\dd x^\prime\dd t+e^{2\sigma(1+\frac{1}{2R})^{2}}\int_{\frac{1}{32}}^{\frac{31}{32}}\int_{\mathbb{R}^{d-1}}\int_{\frac{1}{2}<x_{1}<R}|\tilde{D}_{1}\tilde{v}|^{2}\,\dd x_{1}\dd x^\prime\dd t\\
		%\leq&R^{2}e^{2\sigma(1+\frac{1}{2R})^{2}}\int_{\frac{1}{32}}^{\frac{31}{32}}\int_{\mathbb{R}^{d-1}}\int_{\frac{1}{2}<x_{1}<R}|\tilde{v}|^{2}\,\dd x_{1}\dd x^\prime\dd t+e^{2\sigma(2-\frac{1}{R})^{2}}\int_{\frac{1}{32}}^{\frac{31}{32}}\int_{\mathbb{R}^{d-1}}\int_{\frac{1}{2}<x_{1}<R}|\tilde{D}_{1}\tilde{v}|^{2}\,\dd x_{1}\dd x^\prime\dd t\\
		\leq cR^2e^{2\sigma(1+\frac{1}{2R})^{2}}+e^{2\sigma(2-\frac{1}{R})^{2}}\Theta(R).
	\end{align}
	Combining all the above estimates \eqref{Z0}, \eqref{Z2}, \eqref{Z3} and \eqref{Zleft}, we get
	\begin{align}
		\frac{b}{4}c^{3}R^{2}e^{2\sigma(1+\frac{1}{R})^{2}}\leq ce^{2\sigma(2-\frac{1}{R})^{2}}\Theta(R)+cR^{2}e^{2\sigma(1+\frac{1}{2R})^{2}}.
	\end{align}
	%For large $R$, $cR^{2}e^{2\sigma(1+\frac{1}{2R})^{2}}$ can be absorbed by $\frac{b}{4}c^{3}R^{2}e^{2\sigma(1+\frac{1}{R})^{2}}$. 
	Then we can get
	\begin{equation}
		\frac{b}{8}c^{3}R^{2}e^{2\sigma(1+\frac{1}{R})^{2}}\leq c\Theta(R)e^{2\sigma(2-\frac{1}{R})^{2}}.
	\end{equation}
	From the fact 
	$
	2\sigma(1+\frac1 R)^2-2\sigma(2-\frac1R)^2\geq  -8cR^2,
	%2\sigma(1+\frac{1}{R})^{2}-2\sigma(2-\frac{1}{R})^{2}=-6cR^{2}+12cR\geq-8cR^2,
	$
	we get
	$
	\frac{b}{8}c^{3}R^{2}e^{-8cR^{2}}\leq c\Theta(R).
	$
\end{proof}
\textbf{Step 6: Vanishing on a strip.}
From \eqref{upper bound for derivative}, and $\kappa(x_{1})\sim x_{1}$ for $x_{1}\gg1$, we get 
\begin{align}
	\Theta(R)=&\int_{\mathbb{R}^{d-1}}\int_{R-1<x_{1}<R}\int_{\frac{1}{32}}^{\frac{31}{32}}|\tilde{v}|^{2}+|\tilde{D}_{1}\tilde{v}|^{2}\,\dd t\dd x_{1}\dd x^\prime\\ 
	=&\int_{\mathbb{R}^{d-1}}\int_{R-1<x_{1}<R}\int_{\frac{1}{32}}^{\frac{31}{32}}e^{2\gamma\kappa}e^{-2\gamma\kappa}(|\tilde{v}|^{2}+|\tilde{D}_{1}\tilde{v}|^{2})\,\dd t\dd x_{1}\dd x^\prime\\ \leq&ce^{-\gamma R}R^{3}\lambda\,e^{2\lambda m\delta^{-\frac12}}.
\end{align}

Then for $R$ large, 
\begin{equation}\label{bbb}
	b\leq ce^{8cR^{2}-\gamma R+3\lambda m\delta^{-\frac12}}.
\end{equation}
Since the parameter $\lambda$ is free, we take \begin{equation}
	R=\frac{2Mm\lambda^{\frac{1}{2}}}{(\delta^{\frac{1}{2}}a_{1})^{\frac{1}{2}}c},\,\,\gamma=(\lambda\delta^{\frac{1}{2}}a_{1})^{\frac{1}{2}},\,\,M\geq\frac{4c}{\delta^{\frac{1}{2}}},
\end{equation}
with $M$ and $m$ determined later,
then 
\begin{align}
	8cR^{2}-\gamma R+\frac{3\lambda m}{\delta^{1/2}}
	=\lambda\Big(32c\frac{M^{2}m^{2}}{\delta^{\frac{1}{2}}a_{1}c^{2}}-\delta^{\frac{1}{4}}a_{1}^{\frac{1}{2}}\frac{2Mm}{\delta^{\frac{1}{4}}a_{1}c}+\frac{3m}{\delta^{\frac{1}{2}}}\Big).\label{exponential part}
\end{align}
We expect that \eqref{exponential part} is negative.
Let
\begin{equation}
	32c\frac{M^{2}m^{2}}{\delta^{\frac{1}{2}}a_{1}c^{2}}+\frac{3m}{\delta^{\frac{1}{2}}}<\delta^{\frac{1}{4}}a_{1}^{\frac{1}{2}}\frac{2Mm}{\delta^{\frac{1}{4}}a_{1}^{\frac{1}{2}}c}=\frac{2Mm}{c},
\end{equation}
after dividing by $Mm$, we get
\begin{equation}
	32c\frac{Mm}{\delta^{\frac{1}{2}}a_{1}c^{2}}+\frac{3}{M\delta^{\frac{1}{2}}}<\frac{2}{c}.
\end{equation}Next, we will determine $M,m$ by using the above inequality. The strategy is to select large $M$ such that the second term $\frac{3}{M\delta^\frac12}<\frac1c$. Then fix this $M$, choosing $m$ sufficiently small such that 
$
\frac{32cMm}{\delta^\frac12a_1c^2}<c^{-1}.
$
Therefore, together with \eqref{bbb}, we finally obtain that  $b\leq ce^{-c(m,M,\delta,a_{1})\lambda}$ holds for arbitrary $\lambda$. This implies  that $b=0$. In conclusion, we have $u(x,1)\equiv0$ on a strip $\{x\in\R^d:x_2\in[\frac{m}{2},m]\}$. By the reduction on \textbf{Step 2}
, we finish the proof of this theorem.\end{proof}

\begin{proof}[Proof of Theorem \ref{thm-non}]
Here, we give the sketch of proof.  Let $w(t)=u_1(t)-u_2(t)$, then $w(t)$ satisfies the equation
\begin{equation}
	i\partial_{t}w+\Delta_{A}w+Vw+V_1(x,t)w=0,
\end{equation}
with
\begin{equation}
	V_1(x,t)=\frac{F(u_{1},\bar{u}_{1})-F(u_{2},\bar{u}_{2})}{u_{1}-u_{2}}.
\end{equation}
Following the proof of \cite[Corollay 2.4]{KPV-CPAM},  we obtain
\begin{equation}
	\int_{0}^{1} \|V_1(t)\|_{L^{\infty}(\mathbb{R}^{d}\backslash B_{\rho})}\,\dd t\leq C\sum_{j,l=1}^{2}\big(\|u_{l}\|^{p_{j}-1}_{L_{t}^{p_{j}-1}L^{2}(\mathbb{R}^{d}\setminus B_{\rho})}+\|\partial^{\gamma}u_{l}\|^{p_{j}-1}_{L_{t}^{p_{j}-1}L^{2}(\mathbb{R}^{d}\backslash B_{\rho})}\big)\to0
\end{equation}
as $\rho\to\infty$.
Then, the assumptions in Theorem \ref{thm-P} are satisfied and so it  yields the desired uniqueness result. 
\end{proof}
\textbf{Acknowledgments}. We  thank the anonymous referee and  editor  for their
invaluable comments.  Y. Wang is supported by a Juan de la Cierva fellowship funded by
MICIU/AEI/10.13039/501100011033, under Grant JDC2024 053285-I. J. Zheng was supported by  National key R\&D program of China: 2021YFA1002500 and  NSFC Grant 12271051.

\bibliographystyle{plain}

\end{document}